\documentclass[12pt]{article}

\usepackage{amsmath,amssymb,amsthm,amscd}
\usepackage[colorlinks, linkcolor=blue, anchorcolor=gree, citecolor=red]{hyperref}
\usepackage[numbers,sort&compress]{natbib}
\usepackage{graphicx}

\makeatletter

\newcommand{\Rmnum}[1]{\expandafter\@slowromancap\romannumeral #1@}
\makeatother

\begin{document}

	\newcommand{\Cyc}{{\rm{Cyc}}}\newcommand{\diam}{{\rm{diam}}}
	\newcommand{\Cay}{{\rm Cay}}
	\newtheorem{thm}{Theorem}[section]
	\newtheorem{pro}[thm]{Proposition}
	\newtheorem{lem}[thm]{Lemma}
	\newtheorem{exa}[thm]{Example}
	\newtheorem{fac}[thm]{Fact}
	\newtheorem{cor}[thm]{Corollary}
	\newtheorem{constr}[thm]{Construction}
	\theoremstyle{definition}
	\newtheorem{ex}[thm]{Example}
	\newtheorem{ob}[thm]{Observtion}
	\newtheorem{remark}[thm]{Remark}
	\newcounter{foo}[subsection]
	\newcounter{fooo}[section]
	\newtheorem{step}[foo]{Step}
	\newtheorem{stepp}[fooo]{Step}
	\newcommand{\bth}{\begin{thm}}
		\renewcommand{\eth}{\end{thm}}
	\newcommand{\bex}{\begin{ex}}
		\newcommand{\eex}{\end{ex}}
	\newcommand{\bre}{\begin{remark}}
		\newcommand{\ere}{\end{remark}}
	
	\newcommand{\bal}{\begin{aligned}}
		\newcommand{\eal}{\end{aligned}}
	\newcommand{\beq}{\begin{equation}}
		\newcommand{\eeq}{\end{equation}}
	\newcommand{\ben}{\begin{equation*}}
		\newcommand{\een}{\end{equation*}}
	
	\newcommand{\bpf}{\begin{proof}}
		\newcommand{\epf}{\end{proof}}
	\renewcommand{\thefootnote}{}
	\newcommand{\sdim}{{\rm sdim}}
	
	\def\beql#1{\begin{equation}\label{#1}}
		\title{\Large\bf Perfect codes in 2-valent Cayley digraphs on abelian groups}
		\author{{Shilong Yu$^{1}$,\quad Yuefeng Yang$^{1,}$\footnote{Corresponding author.}, \quad Yushuang Fan$^{1}$,\quad Xuanlong Ma$^{2}$
			}\\[15pt]
			{\small\em $^1$School of Science, China University of Geosciences, Beijing 100083, China}\\
			{\small\em $^2$School of Science, Xi'an Shiyou University, Xi'an 710065, China}\\
		}

	\date{}
	
	\maketitle

	\begin{abstract}
		For a digraph $\Gamma$, a subset $C$ of $V(\Gamma)$ is a perfect code if $C$ is a dominating set such that every vertex of $\Gamma$ is dominated by exactly one vertex in $C$. In this paper, we classify strongly connected 2-valent Cayley digraphs on abelian groups admitting a perfect code, and   determine completely all perfect codes of such digraphs.

	\end{abstract}
	
	
	
	\medskip
	
	\footnote{E-mail addresses: 2019220002@email.cugb.edu.cn (S. Yu), yangyf@cugb.edu.cn (Y. Yang),  fys@cugb.edu.cn (Y. Fan), xuanlma@xsyu.edu.cn (X. Ma).}
	\section{Introduction}
	
 A digraph $ \Gamma$ is a pair $(V(\Gamma), A(\Gamma))$ where $V(\Gamma)$ is a finite  nonempty set of vertices and $A(\Gamma)$ is a set of ordered pairs $(arcs)~(x, y)$ with distinct vertices $x, y$. $\Gamma$ is an  {\em undirected graph} or a {\em graph} if $A(\Gamma)$ is a symmetric relation. A {\em path} of length $r$ from $u$ to $v$ is a finite sequence of vertices $(u=\omega_0,\omega_1,...,\omega_r=v)$ such that $(\omega_{t-1},\omega_t) \in A(\Gamma)$ for $t=1,2,...,r$. A digraph is said to be {\em strongly connected} if, for any two distinct vertices $x$ and $y$, there is a path from $x$ to $y$. We say that a vertex $x$ is {\em adjacent} to $y$ if $(x,y)\in A(\Gamma)$. In this case, we also call $ y$ an {\em out-neighbour} of $ x$, and $x$ an {\em in-neighbour} of $y$. The set of all out-neighbours of $x$ is denoted by $N^{+}(x)$, while the set of in-neighbours is $N^{-}(x)$. A digraph is said to be {\em regular of valency} $k$ if the number of in-neighbour and out-neighbour of all vertices are all equal to $k$.

Let $\Gamma$ be a digraph. A vertex $u$ {\em dominates} a vertex $v$ if either $u=v$ or $(u, v) \in A(\Gamma)$. A set $C \subseteq V(\Gamma)$ is called a {\em dominating set} of the digraph $\Gamma$ if every vertex of $V(\Gamma)$ is dominated by a vertex in $C$. If $C$ is a dominating set of $\Gamma$ such that every vertex of $\Gamma$ is dominated by exactly one vertex in $C$, then $C$ is called a {\em perfect code}.  In some references, a perfect code is also called an {\em efficient dominating set} \cite{DYP,DYP2,DeS} or {\em independent perfect dominating set} \cite{Le}.

Let $G$ be a finite, multiplicatively written, group with identity $e$, and $S$ be a subset of $G$ without identity. A {\em Cayley digraph} of a group $G$ with respect to the set $S$, denoted by $\Cay(G,S)$, is the digraph with vertex set $G$, where $(x,y)$ is an arc whenever $yx^{-1} \in S$.

In the past few years, perfect codes in Cayley graphs have attracted considerable attention.
 In \cite{OPR}, cubic and quartic circulants (that is, Cayley graphs on cyclic groups) admitting a perfect code were classified. \c{C}al{\i}\c{s}kan, Miklavi\v{c} and \"{O}zkan characterized cubic and quartic Cayley
graphs on abelian groups that admit a perfect code \cite{CMO}. Kwon, Lee and Sohn gave necessary and sufficient conditions for the existence of perfect codes of quintic circulants and classified these perfect codes \cite{KLS}. The second author, the fourth author and Zeng \cite{5du} classified all connected quintic Cayley graphs on abelian groups that admit a perfect code, and determined completely all perfect codes of such graphs. \c{C}al{\i}\c{s}kan, Miklavi\v{c} and \"{O}zkan classified the connected cubic Cayley graphs on
generalized dihedral groups which admit a perfect code \cite{CMO22}. For more about perfect codes in Cayley graphs, see for examples \cite{FHZ,Le,Z15,HXZ18,DeS,CWZ,MWWZ,Mo,ZZ20,ZZ21}.

In this paper, we study 2-valent Cayley
digraphs on abelian groups admitting a perfect code. The following theorem classifies 2-valent Cayley digraphs on abelian groups admitting a perfect code.

\begin{thm}\label{main}
	A strongly connected 2-valent Cayley digraph on an abelian group admits a perfect	code if and only if it is isomorphic to $\Gamma_{m,l,h}$, where $\Gamma_{m,l,h}$ is from Construction \ref{construcution} with $0\leq h<m$, $0<l$, $3\mid m$ and $3\mid(l-h)$.
\end{thm}

For the rest of this paper, we always assume that $G$ is an abelian group with identity $0$, written additively. For $a\in G$, let $o(a)$ denote the {\em order} of $a$, that is, the smallest positive integer $m$ such that $ma=0$. An element $a$ is called an {\em involution} if $o(a)=2$. The second main result determines all perfect codes in 2-valent Cayley digraphs on abelian groups.

\begin{thm}\label{main2}
	With the notations in Theorem \ref{main}, suppose that $\Cay(G, S)$ is isomorphic to the digraph in Theorem \ref{main}, where $S=\left\{s, s^{\prime}\right\}$ and $o(s)=m$.
	Then all perfect codes containing identity are exactly $	\bigcup_{i\in [gcd(l-h,m)/3]}D(i)$, where $D\left(i\right)=\left\{(3i+r) s+r s^{\prime} \mid r \in \mathbb{Z}\right\}$.
\end{thm}

The paper is organized as follows.
In Section 2, we construct an infinite family of  Cayley digraphs on abelian groups admitting a perfect code. In Section 3, we give the proofs of Theorems \ref{main} and \ref{main2}.

\section{Constructions}
In this section, we construct an infinite family of 2-valent Cayley digraphs on abelian groups admitting a perfect code. The main idea for the construction is taken from \cite[Construction 3]{YYF16}.

In the remainder of this section, we always assume that $m$ and $l$ are positive integers and $h$ is a nonnegative integer less than $m$. For a positive integer $n$, denote by $[n]$ the set $\{0,1,\ldots,n-1\}$.
\begin{constr}\label{construcution}
	Let $\Gamma_{m,l,h}$ be the digraph with the vertex set $\mathbb{Z}_m\times[l]$ whose arc set consists of $((a,b),(a+1,b))$, $((a,c),(a,c+1))$ and $((a,-1),(a-h,0))$, where $a\in\mathbb{Z}_m$, $b\in[l]$ and $c\in[l-1]$. See Figure \ref{fig:InformativeFigure}.
\end{constr}

For each integer $i$, let $\overline{i}$ denote the residue class $i+n\mathbb{Z}$, and $\hat{i}$ be the minimal nonnegative integer in $\overline{i}$. In $\mathbb{Z}_n$, we write $i$ instead of $\overline{i}$.

\begin{remark}\label{1}
	Let $j\in\mathbb{Z}$. Then there exist $n\in\mathbb{Z}$ and $r\in [l]$ such that $j=nl+r$. 	Similar to \cite[Remark 2.2]{5du}, by Construction \ref{construcution}, we may use $(i,j)$ to denote the vertex $(\hat{i}-nh,r)$ of $V(\Gamma_{m,l,h})$ for all $i\in\mathbb{Z}_m$.

\end{remark}
\begin{figure}[!h]
	\begin{center}
		\includegraphics{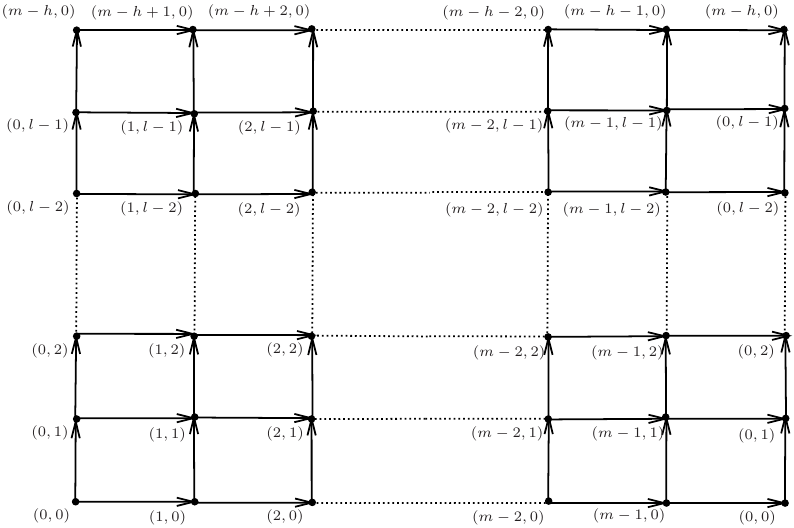}
	\end{center}
	\caption{\label{fig:InformativeFigure} The digraph $\Gamma_{m,l,h}$.}
\end{figure}
The following result is immediate from \cite[Proposition 2.3]{5du}.
\begin{pro}\label{pro1}
	  The digraph $\Gamma_{m,l,h}$ is a Cayley digraph on an abelian group.
\end{pro}

\begin{pro} \label{pro2}
	Let $3\mid m$ and $3\mid(l-h)$. Then $\cup_{r\in[{\rm gcd}(l-h,m)/3]}C_r$ is a perfect code of $\Gamma_{m,l,h}$, where
	\begin{align}
		C_r=\{(3r+j,j)\in\mathbb{Z}_m\times [l]\mid j\in\mathbb{Z}\}.\nonumber
	\end{align}
\end{pro}
\begin{proof}
	For fixed integers $h,l,m$, let $b={\rm gcd}(l-h,m)$. Then $3 \mid b$. Denote $C=\cup_{r\in[b/3]}C_r$. By Remark \ref{1}, $l$ is the minimal positive integer such that $(i,j+l)\in \{{(i',j) \mid i' \in \mathbb{Z}_m}\}$ for each $i \in \mathbb{Z}_m$ and $j \in [l]$. In view of Construction \ref{construcution}, one has $(3r+j+l-h,j)=(3r+j+l,j+l)$,~$(3r+j,j) \in C_r$. Since the order of $l-h$ in $\mathbb{Z}_m$ is $m/b$, we get $|C_r|=ml/b$ for all $r\in [b/3]$.
	
	Suppose $C_r \cap C_s \neq \emptyset$ for distinct $r,s \in [b/3]$. Since $\Gamma_{m,l,h}$ is vertex transitive from Proposition \ref{pro1}, we may assume $r=0$ and $(0,0)\in C_0 \cap C_s$. It follows from Construction \ref{construcution} that $(0,0)=(3s+nl,nl)=(3s+n(l-h),0)$ for some $n \in \mathbb{Z}$. Then $3s \in (l-h)\mathbb{Z}+m\mathbb{Z}=b\mathbb{Z}$, contrary to the fact that $s \in [b/3] \setminus \{0\}$. Thus, $|C|=ml/3$.
	
	We claim that $(i,j) \in C$ if and only if there exists $r \in [b/3]$ such that $\hat{i}-3r \equiv j~({\rm mod}~b)$.
	If $(i,j)\in C$, then $\hat{i}-3r\equiv j~({\rm mod}~b)$ for some $r\in[b/3]$. Now suppose $\hat{i}-3r\equiv j~({\rm mod}~b)$ for some $r\in [b/3]$ and $(i,j)\in\mathbb{Z}_m\times [l]$. Note that there are exactly $m/b$ elements $i\in\mathbb{Z}_m $ satisfying $\hat{i}-3r\equiv j~({\rm mod}~b)$ for each $j\in [l]$ and $r\in [b/3]$. Since $|C|=ml/3$, $(i,j)\in C$ whenever  $\hat{i}-3r\equiv j~({\rm mod}~b)$ for some $r\in [b/3]$.
	
	Note that $C$ is an independent set.
	Let $(i,j)\in(\mathbb{Z}_m\times[l])\setminus C$. Then $\hat{i}-3r\not\equiv j~({\rm mod}~b)$ for $r \in [b/3]$. It follows that $\hat{i}-3r+c\equiv j~({\rm mod}~b)$ for some $r\in [b/3]$ with $c =\pm1 $.
	
	 If $c=-1$, from the claim, then	$\hat{i}-1-3r\equiv j~({\rm mod}~b)$, which implies that $(i-1,j)$ is the unique neighbour of $(i,j)$ in $C$. Now we consider the case $c=1$. If $j\neq 0$, from the claim, then $\hat{i}-3r\equiv j-1~({\rm mod}~b)$, which implies that $(i,j-1)$ is the unique neighbour of $(i,j)$ in $C$. Suppose $j=0$. Since $\hat{i}-3r\equiv-1~({\rm mod}~b)$ and $b\mid l-h$, we get
	$\hat{i}+h-3r\equiv l-1~({\rm mod}~b)$, which implies that $(i+h,l-1)$ is the unique neighbour of $(i,0)$ in $C$.
	
	This completes the proof of this proposition.
\end{proof}

\section{Proofs of Theorems \ref{main} and \ref{main2}}

In this section, we always assume that $\Gamma=\Cay(G,S)$ admits a perfect code $C$, where $S$ is a generating set of
$G$ with $S=\{s,s'\}$. Let $x(i,j)=is+js'$ for all integers $i,j$, where the first coordinate could be read modulo $o(s)$ and the second coordinate could be read modulo $o(s')$. For all $x(i,j)\in G$, we have
\begin{align}
	N^{+}(x(i,j))&=\{x(i+1,j),x(i,j+1)\},\label{neighbors2}\\
	N^{-}(x(i,j))&=\{x(i-1,j),x(i,j-1)\}.\label{neighbors3}
\end{align}

To give the proofs of Theorems \ref{main} and \ref{main2}, we need four auxiliary lemmas.

\begin{lem}\label{lem1}
	 The generating set $S$ has no involution.
\end{lem}
\begin{proof}
	Assume the contrary, namely, $S$ has at least one involution. If $o(s)=o(s')=2$ with $s \neq s' $, then  $\Cay(G,S)$ is an undirected graph with $|G|=4$, contrary to the fact that $3\mid |G|$ from \cite[Proposition 2.1]{Vu} and \cite[Lemma 2.3]{DYP2}. Thus, $S$ has exactly one involution.

	 Without loss of generality, we may assume $o(s')=2$. Since $\Gamma$ is vertex transitive, we may assume $x(0,0) \in C$. By the definition of a perfect code, we have $N^{+}(x(0,0)) \cap  C =\{x(0,1),x(1,0)\} \cap C=\emptyset$ from \eqref{neighbors2}. It follows that $x(0,1),x(1,0) \notin C$. Since $N^{-}(x(1,1))\cap C=\{x(0,1),x(1,0)\}\cap C=\emptyset$ from \eqref{neighbors3}, one gets $x(1,1) \in C$, contrary to the fact that $ N^{-}(x(1,0)) \cap C =\{x(1,1),x(0,0)\}$.

    Thus, $S$  has no involution.
\end{proof}

\begin{lem}\label{lem2}
	 If $x(i,j)\in C$, then $x(i+l,j+l)\in C$ for each integer $l$.
\end{lem}
\begin{proof}
	Since $\Gamma$ is vertex transitive, we may assume $(i,j)=(0,0)$. According to the definition of a perfect code, one gets $N^{+}(x(0,0))\cap C =\emptyset$. By \eqref{neighbors2} and \eqref{neighbors3}, we get $N^{+}(x(0,0))=N^{-}(x(1,1))$, which implies $N^{-}(x(1,1)) \cap C=\emptyset$. It follows from the definition of a perfect code that $x(1,1) \in C$. By induction, one gets $x(l, l) \in C $ for each integer $l$.
	
	This completes the proof of this lemma.
\end{proof}

\begin{lem}\label{lem3}
	Let $x(i,j)\in C$. The following hold:
	\begin{itemize}
		\item[{\rm (i)}] $\{x(i+1,j),x(i+2,j),x(i,j+1),x(i,j+2)\}\cap C=\emptyset$;
		
		\item[{\rm (ii)}] $\{x(i+3,j),x(i,j+3) \} \in C$.
	\end{itemize}
\end{lem}
\begin{proof}
	Since $\Gamma$ is vertex transitive, we may assume $(i,j)=(0,0)$. By Lemma \ref{lem2}, we have $x(1,1) \in C$. Since $C\cap N^{+}(x(h,h))=\emptyset$ for $h\in \{0,1\}$ from \eqref{neighbors2}, we get $x(2,1),x(1,0) \notin C$. The fact $N^{+}(x(1,1))\cap N^{+}(x(2,0))=\{x(2,1)\}$ implies $x(2,0)\notin C$. In view of the symmetry of $\Gamma$, (i) is valid.
	
	By the definition of a perfect code, we have $|N^{-}(x(2,0))\cap C|=1$. Since $x(1,0)\notin C$, from \eqref{neighbors3}, one gets $x(2,-1) \in C$. By Lemma \ref{lem2}, we get $x(3,0) \in C$. In view of the symmetry of $\Gamma$, (ii) is valid.
\end{proof}

\begin{lem}\label{lem4}
	Let $x(0,0)=x(h, l)$ for some $(h, l) \in\{(i, j) \mid 0 \leq   i \leq o(s)$ and $\left.0 \leq j \leq o\left(s^{\prime}\right)\right\}$. Then $3 \mid(l-h)$.
\end{lem}
\begin{proof}
	Since $x(0,0)=x(h,l)$, we have $x(i,j)=x(i-h,j-l)$ for all $x(i,j) \in V(\Gamma)$.
	
	Suppose $3 \nmid l-h$. Without loss of generality, we may assume $l>h$. It follows that $3 \mid l-h+b$ for some $b =\pm1 $. 	Since $\Gamma$ is vertex transitive, from Lemma \ref{lem2}, we may assume $x(i,i) \in C$ for each integer $i$. Then $x(l-h,0)=x(l,l) \in C$. Since $x(0,0) \in C$, by Lemma \ref{lem3}, we have $x(l-h+b,0) \in C$, contrary to the fact that $(x(l-h,0),x(l-h+b,0))$ or $(x(l-h+b,0),x(l-h,0)) \in A(\Gamma)$.
\end{proof}
Now we are ready to give a proof of Theorem \ref{main}.

\begin{proof}[Proof of Theorem~\ref{main}]
	The proof of the sufficiency is straightforward by Propositions \ref{pro1} and \ref{pro2}.
	
	Now we prove the necessity. We may assume that $\Gamma$ admits a perfect code. Let $l$ be the minimal positive integer with $x(0,0)=x(h,l)$ for some $h\in [o(s)]$. Since $x(o(s),0)=x(0,0)=x(h,l)=x(h-l,0)$, from Lemma \ref{lem4}, we get $3\mid o(s)$ and $3 \mid l-h$. 	
	
 Since $\Gamma$ is strongly connected, by Construction \ref{construcution}, $\Gamma $ is isomorphic to $\Gamma_{o(s),l,h}$. By Proposition \ref{pro2}, $\Gamma$ is a Cayley digraph on an abelian group.
	
	This completes the proof of the necessity.
\end{proof}

Now we prove Theorem \ref{main2}.

\begin{proof}[Proof of Theorem~\ref{main2}]
	Without loss of generality,  we may assume that $x(0,0)\in C$. Let $l$ be the minimal positive integer with $x(0,l)\in\langle s\rangle$. Then there exists $h \in [o(s)]$ such that $x(h,l)=x(0,0)$. It follows that $x(i,j)=x(i-h,j-l)$ for all $i,j\in\mathbb{Z}$.
	
	Let $H$ be a subset of $\langle s\rangle$ such that $is\in H$ whenever there exists integer $j$ satisfying $x(j,j)=x(i,0)$. We claim that $H$ is a subgroup of $\langle s\rangle$. Since $0\in H$, $H$ is not empty. Let $is,i's\in H$. Then there exist integers $j, j'$ such that $x(j,j)=x(i,0)$ and $x(j',j')=x(i',0)$. It follows that $x(j-j',j-j')=x(i-i',0)$, and so $(i-i')s\in H$. Thus, our claim is valid.
	
	Note that $is\in H$ if and only if there exists integer $j$ satisfying $l\mid j$ and $x(j,j)=x(i,0)$. Since $x(l,l)=x(l-h,0)$, from the claim, one has $H=\langle(l-h)s\rangle$. Since $H$ has ${\rm gcd}(l-h,m)$ cosets in $\langle s\rangle$, by Lemma \ref{lem3}, we get $x(3r,0)+H\subseteq C$ for all $1\leq r \leq {\rm gcd}(l-h,m)/3-1$. It follows from Lemma \ref{lem2} and Proposition \ref{pro2} that the desired result is valid.
\end{proof}

	\bigskip
	\noindent \textbf{Acknowledgements}~~
	We are grateful to the referees for useful comments and suggestions. Y. Yang is supported by the National Natural Science Foundation of China (12101575, 52377162) and the Fundamental Research
	Funds for the Central Universities (2652019319). X. Ma's research is supported by National Natural Science Foundation of China (Grant No. 11801441, 12326333, 61976244), and Shaanxi Fundamental Science Research Project for Mathematics and Physics (Grant No. 22JSQ024).
	
	\section*{Data Availability Statement}
	
	Data sharing not applicable to this article as no datasets were generated or analysed during the current study.

	\section*{Declarations}
	
	{\bf Conflict of interest} The authors declare that they have no conflict of interest.

\end{document}